\documentclass[a4paper, 12pt, oneside]{amsart} % ECinProb uses even 10pt (!)

\usepackage[paper=a4paper,left=25mm,right=25mm,top=30mm,bottom=30mm]{geometry}
\usepackage{amssymb}
\usepackage[english]{babel}
\usepackage{graphicx}
\usepackage{amsmath}
\usepackage{amsthm}
\usepackage{bbm}
\usepackage{xcolor} %Schriftfarben
\usepackage[textwidth=3.5cm]{todonotes}
\usepackage{color}
\definecolor{Myred}{rgb}{0.85,0,0.3}%  Wojciech

\newtheorem{theorem}{Theorem}[section]
\newtheorem{proposition}[theorem]{Proposition}
\newtheorem{lemma}[theorem]{Lemma}
\newtheorem{corollary}[theorem]{Corollary}
\theoremstyle{definition}

\newtheorem{example}[theorem]{Example}
\newtheorem{remark}[theorem]{Remark}

\theoremstyle{remark}

\newcommand{\p}{\mathbbm{P} }  
\newcommand{\e}{\mathbbm{E} }

\newcommand{\PP}{\mathbb{P}}
\newcommand{\IndFun}{\textbf{1}_}

\usepackage{hyperref} 
\hypersetup{
pdfpagemode=UseNone,
pdfstartview=FitH,
pdfdisplaydoctitle=true,
pdftitle={On recurrence of the  multidimensional Lindley process},
pdfauthor={Wojciech Cygan, Judith Kloas},
pdflang=en-US
}
\title[On recurrence of the multidimensional Lindley process]{On recurrence of the multidimensional\\ Lindley process}

\author{Wojciech Cygan}
\thanks{The research supported by: Austrian Science Fund projects FWF P24028 and W1230, NAWI Graz, and National Science Centre (Poland): grant 2015/17/B/ST1/00062.}

\author{Judith Kloas}

					 %Stable processes

\address{							
		Wojciech Cygan\\
		Instytut Matematyczny\\
		Uniwersytet Wroc{\l}awski\newline
		Pl. Grunwaldzki 2/4\\
		50-384 Wroc{\l}aw, Poland \newline
		\& Institut f\"{u}r Diskrete Mathematik\\
 		Technische Universit\"{a}t Graz\newline
 		Steyrergasse 30\\
 		8010 Graz, Austria}
\email{wojciech.cygan@uwr.edu.pl}

\address{
		Judith Kloas\\
		Institut f\"{u}r Diskrete Mathematik\\
 		Technische Universit\"{a}t Graz\newline
 		Steyrergasse 30\\
 		8010 Graz, Austria}
\email{kloas@math.tugraz.at}


\begin{document}

\begin{abstract}
A Lindley process arises from classical studies in queueing theory and it usually reflects waiting times of customers in  single server models. In this note we study recurrence of its higher dimensional counterpart under some mild assumptions on the tail behaviour of the underlying random walk. There are several links between the Lindley process and the associated random walk and we build upon such relations. We apply a method related to discrete subordination for random walks on the integer lattice together with various facts from the theory of fluctuations of random walks. 
\end{abstract}

\maketitle

\section{Introduction}
Let $(Y_n)_{n\geq1}$ be a sequence of independent and identically distributed random variables with common distribution $\mu$.
A Lindley process (LP) is a discrete time stochastic process $(W_n)_{n\geq 0}$ defined recursively by
\begin{align}\label{UniVarDef}
W_0=w_0\geq 0\quad \text{and} \quad W_{n}=\max\{ W_{n-1}-Y_{n}, 0\} , \quad \text{for }n\geq 1. 
\end{align}
The random variable $W_n$ may be interpreted as the waiting time of the $n$-th client in a single server queue, where customers arrive randomly and are served within a random amount of time. 
More precisely, if we let $U_n$ to be the service time of the $n$-th client and $T_n$ to be the time between the arrival of the $(n-1)$-th client and the $n$-th client, then the relation between $W_{n+1}$ and $W_n$ is exactly $W_{n+1}= \max\{W_{n}+U_n-T_n,0\}$. Hence \eqref{UniVarDef} holds with $Y_n = T_n-U_n$ which must be i.i.d. We mention that LP comes up in many different places in the queuing theory, see \textsc{Asmussen} \cite{as00} for some examples.

The process $(W_n)$ may be also regarded as a Markov chain on the state space $[0,\infty )$ with one-step transition probabilities given by $ p (w_0, [0,w]) = \PP (W_1 \leq w \, |\, W_0=w_0) = \mu \left([w_0-w,\infty )\right)$, for $ w\geq 0$. 
Let $S_n=Y_1+\cdots +Y_n$ be the associated random walk.  
Relation \eqref{UniVarDef} reveals that the LP which starts at $0$ obeys the same transition rules as the random walk $(S_n)$, except the times when $(S_n)$ crosses its successive maximal levels, since at these moments $(W_n)$ stays at $0$. 
In other words, the return times to $0$, denoted by $T_W(k)$, $k\geq 0$, for the process $(W_n)$ started at $0$ coincide with the ascending ladder epochs of the random walk $(S_n)$. Let us recall that the (non-strict) ascending ladder epochs
are defined as
\begin{align*}
\bar{\tau}(0)&=0, \quad \bar{\tau}(k+1)=\inf\{n>\bar{\tau}(k):S_n\geq S_{\bar{\tau}(k)}\},\quad \text{for } k\geq 0,
\end{align*}
where $S_0=0$ and we use the convention that $\inf \emptyset =\infty$. It is straightforward to check that $T_W(k) = \bar{\tau} (k)$.
There are also more connections like this and one of the most significant is that, given $W_0=0$, the random variable $W_n$ has the same distribution as $M_n = \max\{0, S_1 ,\ldots , S_n\}$. 
All the mentioned facts bear a lot of fruitful consequences and we exploit them repeatedly in our paper.
The main aim of the article is to present sufficient criteria for recurrence of the multidimensional counterpart of the LP.

%\subsection*{Recurrence of the one-dimensional lattice LP}
We briefly state the well-known facts about recurrence of the LP in the one-di\-mensional case.  
Recall that an essential class for a Markov chain is a subset of the state space which is irreducible and absorbing. 
Given $\PP (Y_1>0)>0$ there is only one essential class for $(W_n)$ and it contains all the states that can be visited after the process reached $0$. Thus to study its recurrence it suffices to concentrate on the behaviour at the origin.
%In particular, if we assume that $\mathrm{supp}\, \mu \subseteq \mathbb{Z}$ and $\mathrm{gcd}\left( \mathrm{supp}\, \mu \right) = 1$ then $(W_n)$ is irreducible on $\mathbb{N}_0 = \mathbb{N}\cup \{0\}$.

We recall from \textsc{Feller} \cite[Ch. XII, Sec. 2, Theorem~1]{fe71II} that there are three types of random walks:  
 $(S_n)$ is either oscillating, then $\liminf_{n\rightarrow \infty} S_n=-\infty$ and $ \limsup_{n\rightarrow \infty} S_n=\infty$;
or it has a positive drift with $\lim_{n\rightarrow \infty} S_n =\infty$; or
it has a negative drift meaning that $\lim_{n\rightarrow \infty} S_n =-\infty$.
In the first two cases we have $\PP(\bar{\tau }(1) <\infty)=1$, whereas in the negative drift case $\p(\bar{\tau }(1)<\infty)<1$.
By the correspondence between the ladder epochs of $(S_n)$ and the return times of $(W_n)$, we conclude that $(W_n)$ is recurrent if and only if $\p(\bar{\tau}(1)<\infty)=1$.
% and it is positive recurrent if and only if $ \e(\tau(1))<\infty$, which holds in the presence of positive drift. 
Therefore $(W_n)$ is recurrent if and only if $(S_n)$ is oscillating or if it has a positive drift and the following dichotomy holds true: 

\vspace*{0,1cm}
\noindent 1) The process $(W_n)$ is null recurrent
if and only if $(S_n)$ is oscillating. Then $\bar{\tau} =\bar{\tau} (1)$ has infinite first moment, cf. \textsc{Gut} \cite[Theorem 9.1]{gu88}. It happens if $\e Y_1=0$ or if $\mu$ is symmetric.

\vspace*{0,1cm}
\noindent 2) The process $(W_n)$ is positive recurrent if and only if $(S_n)$ has a positive drift. In this case $\e \bar{\tau} $ is finite and $W_n$ converges weakly to the random variable $M_\infty =\sup \{S_0,S_1,\ldots \}$ which is finite a.s. This holds in particular if $\e |Y_1| <\infty$ and $\e Y_1 >0$.

\vspace*{0,1cm}
 We observe that for a general distribution $\mu$ on $\mathbb{R}$ and the associated LP with an arbitrary initial random variable $W_0\geq 0$ which is independent of $(Y_n)$ we have equality in law $W_n = \max\{ M_{n-1}, W_0+S_n\}$, for $n\geq 1$. This in turn implies that, given $\e(Y_1)>0$, $W_n\to M_{\infty}$ in law and thus the distribution of $M_{\infty}$ is the unique stationary measure for $(W_n)$, cf. also \textsc{Diaconis and Freedman} \cite[Theorem 4.1]{Diaconis}.

As already mentioned, the LP comes up naturally in the framework of single server queues and thus it was extensively studied over the past decades, see e.g. the seminal paper by \textsc{Kendall} \cite{Kendall} with references therein and cf. also the books by \textsc{Feller} \cite{fe71II}, \textsc{Borovkov} \cite{Borovkov1976} and \textsc{Asmussen} \cite{as00}. \textsc{Lindley} \cite{Lindley} was the first who investigated the limit behaviour of $(W_n)$ and discovered its connections with the Wiener-Hopf integral equations. More recently, asymptotics of the return probabilities of $(W_n)$ were computed by \textsc{Essifi, Peign\'{e} and Raschel} \cite{espera13}. 

The LP may be also viewed as a random walk with a certain barrier at zero and in this spirit we mention the reflected random walk (RRW), denoted by $(X_n^x)_{n\geq 0}$, which is defined analogously to $(W_n)$ but instead of the maximum function in \eqref{UniVarDef} one sets $X_0^x = x\geq 0$ and $X_n^x = |X_{n-1}^x-Y_n|$, for $ n\geq 1$.
There is an obvious and striking resemblance between the two processes and in this note we take advantage of this aspect. In particular, the question of recurrence of RRW received much attention in the literature, see \textsc{Peign\'{e} and Woess} \cite{pewo08} with references therein and \textsc{Kloas and Woess} \cite{klwo16} for a treatment of the multidimensional case.  From that perspective, one can use powerful methods related to stochastic dynamical systems in order to study recurrence of various processes. We partially apply such techniques to obtain a result concerning positive recurrence of the multidimensional LP in the final section.

The multidimensional counterpart of the LP arises from the studies on many server queueing models which were initiated by \textsc{Kiefer and Wolfowitz} \cite{Kiefer_Wolfowitz}. 
In this note we aim at finding sufficient conditions for recurrence of the multidimensional LP as well as for a process of which some coordinates are Lindley processes and the other are ordinary random walks. We focus mainly on the two-dimensional lattice case but we also present a satisfactory result for higher dimensions. 
More precisely, the paper is organized as follows: Section \ref{Sec_Walks} is devoted to the study of the asymptotic behaviour of a given random walk on integers which is evaluated at some random stopping times that are ladder epochs of a second independent random walk. Further, we take advantage of the result and construct a pair of examples of random walks with infinite second moment and investigate their recurrence.
In section \ref{Sec_MLP} we treat the two-dimensional LP in the lattice quadrant and investigate its recurrence under various assumptions on the tail behaviour of the underlying random walk. Among other methods, we 
apply the asymptotics obtained in Section \ref{Sec_Walks}. In the last paragraph we use a technique of local contractivity, which is related to stochastic dynamical systems, to study positive recurrence of the LP in higher dimensions. 

\vspace*{0,2cm}
\noindent
\textit{Notation.} We use the standard notation: $f(x)\sim g(x)$, as $x\to a$ if $f(x)/g(x)\to 1$, as $x$ tends to $a$. Similarly we write $f(x)=o(g(x))$, as $x\to a$ if $f(x)/g(x)\to 0$, as $x$ tends to $a$.

\section{Subordination tools for random walks} \label{Sec_Walks}

In this section we focus on the local asymptotic behaviour of time-changed (subordinated) random walks on the integer numbers. 
We consider an increasing random walk which is responsible for a random change of time. We usually assume that this random walk has heavy tails and 
in our studies on the LP it is supposed to coincide with a sequence of ladder epochs of some other random walk.
We then use the theory of regular variation to find exact asymptotics for the tails of the accordingly time-changed random walk. This enables us to present a vast class of examples of subordinated random walks and to handle the question concerning their recurrence. Although this topic is interesting in itself, our primary goal is to apply results of this section to find some criteria for the recurrence of the two-dimensional process of which one coordinate is a LP whereas the second is a random walk, cf. Theorem \ref{the: recurrence of (W_n,Z_n)}. 

Let $S_n=Y_1+\ldots +Y_n$ be an oscillating random walk such that $S_0=0$. We always assume that the distribution $\mu$ of the increment $Y$ is supported by $\mathbb{Z}$.
% which is equivalent to $\sum_{n=1}^\infty n^{-1}\PP(S_n>0) = \sum_{n=1}^\infty n^{-1} \PP(S_n\leq 0)=\infty$.
Since $(S_n)$ is oscillating, the first strict ascending ladder epoch $\tau =\tau (1) =\min \{ n\geq 1: S_n>0 \}$ is well-defined.
Following \textsc{Vatutin and Wachtel} \cite{vawa08}, for $\alpha ,\beta \in \mathbb{R}$ we consider the set 
\begin{align}\label{ali: definition of A}
\mathcal{A} = \{0<\alpha<1;\, |\beta|<1 \} \cup \{1<\alpha <2;\, |\beta |\leq 1\} \cup \{\alpha=1,2;\, \beta=0\}.
\end{align}
For $(\alpha ,\beta )\in \mathcal{A}$ we write $Y\in D(\alpha ,\beta )$ if the distribution of $Y$ belongs to the domain of attraction of the stable law with characteristic function
\begin{align*}
\Phi (\xi) = \exp \left\{ -c|\xi|^{\alpha} \left( 1-i\beta \frac{\xi}{|\xi|}\tan \frac{\alpha \pi }{2}\right) \right\},
\end{align*}
 for $c>0$.
If $1<\alpha \leq 2$ we assume that $\e(Y) =0$. It is known by \textsc{Doney} \cite{do95} that if $Y\in D(\alpha ,\beta)$ then 
\begin{align}\label{Rogozin_cond}
\p(S_n>0)\rightarrow \rho\in (0,1),\quad n\to \infty ,
\end{align}
where the parameter $\rho $ is given by
\begin{align}\label{ali: definition of rho}
\rho&= \frac{1}{2}+\frac{1}{\pi \alpha} \arctan\left(\beta \tan\frac{\pi \alpha}{2}\right).
\end{align}
Moreover, condition \eqref{Rogozin_cond} is equivalent to the existence of a slowly varying (at infinity) function $\ell$ such that
	\begin{align}\label{Tau_asympt}
		\p(\tau>n)\sim \frac{1}{\Gamma(\rho)\Gamma(1-\rho)n^\rho\ell(n)},\quad n\to \infty.
	\end{align}
Recall that a function $f$ is regularly varying of index $\gamma$ at infinity if 
$\lim_{x\rightarrow \infty} f(\lambda x)/f(x)= \lambda^\gamma$, for all $\lambda>0$,
and $f$ is called slowly varying if $\gamma=0$. Equation \eqref{Tau_asympt} means that $\tau $ belongs to the domain of attraction of the one-sided stable law of index $\rho$.
According to \cite[Theorem 3]{vawa08} and \cite[Theorem 10]{vawa2} for $\alpha =1$ or $\beta = -1$ we also have the following local result
\begin{align}
\p(\tau=n)\sim \frac{\rho }{\Gamma\left(\rho\right)\Gamma(1-\rho)n^{\rho+1}\ell(n) },\quad  n\rightarrow\infty, \label{ali: Vatutin Wachtel result}
\end{align}
with the same slowly varying function $\ell$ as in \eqref{Tau_asympt}. 

We study the local asymptotic behaviour of a random walk which is evaluated at ladder epochs of the random walk $(S_n)$. More precisely, we consider a finite range and centered random walk $(Z_n)$ on $\mathbb{Z}$ (i.e. the support of the law of $Z_1$ is bounded and $\e Z_1 =0$) and we look more closely at the tail decay of the random variable $Z_{\tau }$, where $\tau $ is the first strictly ascending ladder epoch of $(S_n)$.
The proof of the theorem below is based on the similar result obtained in \textsc{Bendikov and Cygan} \cite{cygan1} for the Green function of the subordinated random walk in $\mathbb{Z}^d$ but it requires numerous improvements and adjustments to the present setting.
To our best knowledge, this is the first result of this type in the centred but not necessarily symmetric case.

We emphasise that the scope of the theorem is wider than it is stated. One can consider an arbitrary increasing random walk $(\eta_n)$ on non-negative integers and then a new subordinated random walk $(Z_{\eta_n})$. The result is applicable given that the increments of $(\eta_n)$ behave locally as in \eqref{ali: Vatutin Wachtel result}. We obtain the local behaviour of the subordinated random walk without any assumption on the structure of the distribution of $\eta_1$, cf. \textsc{Bendikov, Cygan and Trojan} \cite{BCT} for the detailed discussion on the asymptotic behaviour of subordinated random walks under the assumption that the Laplace transform of $\eta_1$ is governed by a Bernstein function.

\begin{theorem}\label{THM_main}
Suppose that $(S_n)$ is an oscillating random walk such that its increment $Y\in D(\alpha,\beta)$.
Let $\tau$ be the first strict ascending ladder epoch of $(S_n)$. Assume that $(S_n)$ is independent of $(Z_n)$, then
\begin{align}\label{symm_asymt_Z_tau}
\p \left(Z_{\tau}=x\right)&\sim  \frac{C(\rho)}{|x|^{2\rho+1}\ell(|x|^2)} , \quad \text{as } |x|\rightarrow\infty,
\end{align}
where $\ell$ is the slowly varying function from (\ref{ali: Vatutin Wachtel result}) and
\begin{align}\label{C_rho}
C(\rho)=  \frac{\rho (2\sigma ^2)^{\rho }\Gamma\left(\rho+\frac{1}{2}\right)}{\sqrt{\pi}\Gamma(\rho)\Gamma(1-\rho)},\quad \textrm{with } \sigma ^2 =\mathrm{Var}(Z_1).
\end{align}
\end{theorem}

\begin{proof}
We set $p_n(x)=\p(Z_n=x)$ and write
\begin{align*}
\p(Z_{\tau}=x)=
\sum_{n=1}^{\left[|x|^{5/3}\right]} p_n(x)  \p(\tau=n)\ + \!\!\!\!\!\!
\sum_{n>\left[|x|^{5/3}\right]} \!\!\! \!  p_n(x) \p(\tau=n)
= I_1(x)+I_2(x).
\end{align*}
Let $\overline{p}_n(x)= (\sqrt{2\pi n}\sigma)^{-1} e^{-|x|^2/(2\sigma^2 n)}$ and $E(n,x)=p_n(x)-\overline{p}_n(x)$. 
Applying \textsc{Lawler and Limic} \cite[Theorem 2.1.1]{lali10} (see the discussion following Proposition 2.1.2), we get that for a centered irreducible and aperiodic random walk in $\mathbb{Z}^d$ with finite third moment there is some $C>0$ such that
\begin{align}
|E(n,x)|&\leq C n^{-\frac{d+1}{2}}, \quad n\geq1.\label{ali: upper bound for E(n,x)}
\end{align}
We decompose $I_2(x)$ into two parts 
\begin{align*}
I_2(x)=\sum_{n>\left[|x|^{5/3}\right]} \overline{p}_n(x) \p(\tau=n) +\sum_{n>\left[|x|^{5/3}\right]} E(n,x)  \p(\tau=n) =I_{21}(x)+I_{22}(x),
\end{align*}
and first we establish that 
$I_{22}(x)=o\left( |x|^{-2\rho-1}/\ell(|x|^2) \right)$.
Our assumptions combined with \eqref{ali: Vatutin Wachtel result} and \eqref{ali: upper bound for E(n,x)} for $d=1$ imply that for some $C>0$
\begin{align*}
I_{22}(x)\leq C\!\!\! \sum_{n>\left[|x|^{5/3}\right]} \frac{1}{ n^{\rho+2}\ell(n)} \sim C\int_{|x|^{5/3}}^\infty \frac{1}{ t^{\rho+2}\ell(t)} \mathrm{d}t, \quad \text{as } |x|\rightarrow \infty.
\end{align*}
By \textsc{Bingham, Goldie and Teugels} \cite[Proposition 1.5.10]{bi87}, we have
\begin{align*}
 \int_{|x|^{5/3}}^\infty \frac{1}{t^{\rho+2} \ell(t)} \mathrm{d}t &\sim
  \frac{1}{( \rho+1) |x|^{5(\rho+1)/3}\ell\left(|x|^{5/3}\right)} , \quad \text{as } |x|\rightarrow \infty
\end{align*}
and whence, for $|x|$ large enough,
\begin{align*}
I_{22}(x)|x|^{2\rho+1} \ell(|x|^2)
&\leq 
C\frac{1}{(\rho+1)|x|^{(2-\rho)/3}} \frac{\ell(|x|^2)}{\ell\left(|x|^{5/3}\right)}.
\end{align*}
We show that the right hand side of the last inequality tends to 0. To do so we find an upper bound for the last fraction with the slowly varying function. For that we apply
Potter bounds \cite[Theorem 1.5.6]{bi87} which assure that for any $\varepsilon >0$ there exists $X\geq 0$ such that
\begin{align*}
\frac{\ell(|x|^2)}{\ell\left(|x|^{5/3}\right)}\leq 2 \max \left\{ \left( \frac{|x|^2}{|x|^{5/3}}\right)^{\varepsilon}, \left( \frac{|x|^2}{|x|^{5/3}}\right)^{-\varepsilon}  \right\},\quad |x|\geq X.
\end{align*}
If we choose $\varepsilon =1$ we obtain that
$\ell(|x|^2)\leq 2|x|^{1/3}\ell\left(|x|^{5/3}\right)$, for $|x|$ large enough, and therefore
\begin{align*}
I_{22}(x)|x|^{2\rho+1} \ell(|x|^2)
& \leq
C\frac{1}{(\rho+1)|x|^{(1-\rho)/3}} \rightarrow 0, \quad \text{as } |x|\rightarrow \infty,
\end{align*}
as desired.
Next, with $I_{21}(x)$ we proceed as follows. For $\vert x \vert \rightarrow \infty$,
\begin{align*}
I_{21}(x)&\sim C_1 \sum_{n>\left[|x|^{5/3}\right]} e^{-\frac{|x|^2}{2\sigma^2 n}} \frac{1} {n^{\rho+3/2}\ell(n)}
\sim C_1\int_{|x|^{5/3}}^\infty  e^{-\frac{|x|^2}{2\sigma^2 t}} \frac{1}{ t^{\rho+3/2}\ell(t)}\mathrm{d}t,
\end{align*}
where $C_1=\rho (\sigma \sqrt{2\pi}\Gamma(\rho)\Gamma(1-\rho))^{-1}$. By a suitable change of variables we have
\begin{align}
\int_{|x|^{5/3}}^\infty  e^{-\frac{|x|^2}{2\sigma^2 t}} \frac{1}{ t^{\rho+3/2}\ell(t)}\mathrm{d}t
= \frac{(2\sigma ^2)^{\rho+\frac{1}{2}}}{|x|^{2\rho+1} \ell(|x|^2)}  \int_0^{|x|^{1/3}/(2\sigma^2)}\!\!\!\! e^{-s} s^{\rho-\frac{1}{2}}\frac{\ell(|x|^2)   }{\ell\left(|x^2|/(2\sigma^2 s)\right)}\mathrm{d}s.\label{Int111}
\end{align}
We choose an arbitrary $\varepsilon \in (0,(2\rho +1)/2)$. By Potter bounds we get that for $|x|$ big enough, 
\begin{align*}
\ell(|x|^2)\leq 2 \max\left\{(2\sigma^2 s)^{-\varepsilon} ,(2\sigma^2 s)^{\varepsilon}\right\}\ell\left(|x|^2/(2\sigma^2 s)\right)
\end{align*}
and this implies
\begin{align*}
e^{-s} s^{\rho-\frac{1}{2}}\frac{\ell(|x|^2)   }{\ell\left(|x^2|/(2\sigma^2 s)\right)} \leq 2\sigma^{-2\varepsilon}e^{-s}s^{\rho +1/2-\varepsilon -1}\mathbf{1}_{(0,1)}(s) + 
2\sigma^{2\varepsilon}e^{-s}s^{\rho +1/2+\varepsilon -1}\mathbf{1}_{[1,\infty)}(s).
\end{align*}
With this estimate we are allowed to apply the dominated convergence theorem to the last integral in \eqref{Int111} which thus converges to $\Gamma (\rho +1/2)$. Hence
\begin{align*}
I_{21}(x)&\sim   \frac{\rho (2\sigma^{2})^{\rho }\, \Gamma\left(\rho+\frac{1}{2}\right)}{\sqrt{\pi}\, \Gamma(\rho)\Gamma(1-\rho)}\frac{1}{ |x|^{2\rho+1}\ell(|x|^2)},\quad \textrm{as}\ |x|\to \infty.
\end{align*}

We are left to show that 
$I_{1}(x)=o\left(|x|^{-2\rho-1}/ \ell(|x|^2) \right)$.
Here we use the assumption that the random walk $(Z_n)$ has finite range. The Gaussian upper bound of \textsc{Alexopoulos} \cite[Theorem 1.8]{al02} yields that there is $c>0$ such that $p_n(x) \leq cn^{-1/2}e^{-cx^2/n}$, for all $n\in \mathbb{N}$ and $x\in \mathbb{Z}$. We have 
\begin{align*}
e^{-\frac{cx^2}{n}} =e^{-\frac{cx^2}{2n}} \cdot e^{-\frac{cx^2}{2n}}  \leq e^{-\frac{c|x|^{1/3}}{2}}\cdot  e^{-\frac{cx^2}{2n}} \leq e^{-\frac{c|x|^{1/3}}{2}}, \quad \mathrm{for} \ n\leq \big[|x|^{5/3}\big]
\end{align*} 
and therefore we obtain
\begin{align*}
I_{1}(x)
\leq c\, e^{-\frac{c|x|^{1/3}}{2}} \sum_{n=1}^{\left[|x|^{5/3}\right]} n^{-1/2} e^{-\frac{cx^2}{2n}} \p(\tau=n)\leq c\, e^{-\frac{c|x|^{1/3}}{2}} .
\end{align*}
%Since $n^{-\frac{1}{2}}e^{-cx^2/(2n)}\leq 1$, the above sum is also bounded by one and thus $I_1(x)\leq c\, e^{-c|x|^{1/3}/2}.$
Observe that $x^{-\nu}\ell (|x|^2)$ tends to $0$ for any $\nu >0$. We conclude that
\begin{align*}
|x|^{2\rho+1} \ell(|x|^2) I_1  (x)\leq c |x|^{2\rho+1+\nu}e^{-\frac{c|x|^{1/3}}{2}}
\end{align*}
and the last quantity tends to 0 as $|x|\rightarrow \infty$, what finishes the proof.
\end{proof}

\begin{corollary}\label{pr: Z is in the domain of attraction}
Under the assumptions of Theorem \ref{THM_main}, $Z_\tau \in D(2\rho ,0)$.
%if the random walk $(Z_n)$ is symmetric then 
%the random variable $Z_{\tau}$ belongs to the domain of attraction of the symmetric stable law of index $2\rho \in (0,2)$. 
\end{corollary}

\begin{proof}
Let $F(x)=\p(Z_{\tau}\leq x)$.
%, then symmetry of $Z$ implies that $F(-x)=1-F(x)+\p(Z_{\tau}=x)$.
By Theorem \ref{THM_main}, as $x\rightarrow \infty$,
\begin{align*}
1-F(x)=\sum_{k>x}  P(Z_\tau =k)
 \sim C(\rho) \sum_{k>x}  \frac{1}{k^{2\rho+1}\ell(k^2)}
  \sim C(\rho) \int_{x}^\infty  \frac{1}{t^{2\rho+1}\ell(t^2)}\mathrm{d}t.
\end{align*}
Hence, by \cite[Prop. 1.5.10]{bi87},
$1-F(x)\sim C(\rho) /(2\rho \, x^{2\rho}\ell(x^2)) $ at infinity.
Asymptotics \eqref{symm_asymt_Z_tau} are symmetric in $x$ and whence one easily shows that
$F(-x)/(1-F(x))$ tends to $1$ as $x$ goes to infinity.
%F(-x) = \sum_{k\leq -x}P(Z_\tau =k) = \sum_{k\geq |x|}P(Z_\tau =-k)\sim \sum_{k\geq |x|}P(Z_\tau =k)\sim 1-F(|x|)
We conclude that $ 1-F(x)+F(-x)\sim C(\rho)/(\rho \, x^{2\rho} \ell(x) )$ at infinity.
The conditions of \textsc{Gnedenko and Kolmogorov} \cite[\S 35, Thm. 2]{gnko69} are fulfilled and we obtain that $Z_{\tau}$ belongs to the domain of attraction of the stable law of index $2\rho$. Since $F(-x)/(1-F(x)+F(-x))$ tends to $1/2$ as $x$ goes to infinity, the skewness parameter $\beta$ equals $0$.
\end{proof}

We next present a variety of examples of random walks on $\mathbb{Z}$ which are constructed according to the discussed procedure of the random change of time. 
For that reason we consider a sequence of strict ascending ladder times $\tau (n)$ which are defined via
\begin{align*}
\tau(0) =0,\quad \tau (k+1) = \inf \{ n>\tau(k) :\, S_n > S_{\tau(k)}\}.
\end{align*}
As we proved that $Z_ {\tau}\in D(2\rho ,0)$, we get that $\e(|Z_{\tau}|^\gamma)<\infty$, for all $\gamma<2\rho$.
First we handle the case $\rho \neq 1/2$.

\begin{proposition} \label{pro: for which rho is Z_tau recurrent/tranisent}
If $(Z_n)$ is symmetric then under the conditions of Theorem \ref{THM_main}, the random walk $(Z_{\tau(n)})$ is transient if $0<\rho<\frac{1}{2}$ and recurrent 
if $\frac{1}{2} <\rho <1$.
\end{proposition}
\begin{proof}
If $\frac{1}{2}<\rho<1$ then $\e(| Z_\tau |)<\infty$ and by symmetry we have $E(Z_\tau )=0$ which yields recurrence.
If $0<\rho<\frac{1}{2}$ we set $F(x)=\p(Z_\tau \leq x)$ and let $H(x)=1-F(x)+F(-x)$ be the tail function. Then by symmetry and Theorem \ref{THM_main}, for some $C>0$,
\begin{align*}
\frac{H(x)}{1-F(x)}&=\frac{2(1-F(x))+P(Z_\tau =x)}{1-F(x)} \sim 2+\frac{C(\rho) 2 \rho \, x^{2\rho} \ell(x^2)}{Cx^{2\rho+1}\ell(x^2)} \rightarrow 2, \quad \text{as } x\rightarrow \infty.
\end{align*}
Thus $H(x) \sim 2(1-F(x)) \sim C(\rho)/(\rho \, x^{2\rho} \ell(x^2))$ at infinity.

Let $\phi(t)=\e(e^{itZ_{\tau}})$ be the characteristic function of $Z_{\tau}$. By symmetry it is a real and even function. 
The result by \textsc{Pitman} \cite[Theorem 1]{pi68} implies that, as $t\to 0$,
\begin{align*}
1-\phi(t)&\sim\frac{\pi\, H(t^{-1})}{2\Gamma(2\rho) \sin(\rho\pi)}
\sim C_1(\rho) t^{2\rho}\ell (t^{-2}),\ \textrm{with }
C_1(\rho)=C(\rho)\frac{\pi }{4\rho\Gamma(2\rho) \sin(\rho\pi)} .
\end{align*}
To prove transience we apply the Chung and Fuchs criterion \cite{chfu51}, see also \textsc{Spitzer} \cite[Ch. 2, Sec. 8, T2]{sp76}. Since the random walk $(Z_{\tau (n)})$ is aperiodic (according to \cite[Ch. 1, Sec. 2, Def. D2]{sp76}), $\phi(\theta)=1$ if and only if $\theta= 2k\pi$, $k\in \mathbb{Z}$, and whence it suffices to prove that $\int_{0}^{\epsilon} (1-\phi(t))^{-1}\mathrm{d}t$ is finite for small $\epsilon >0$ which in view of the previous formula is equivalent to the convergence of $\int_{0}^{\epsilon}(t^{2\rho} \ell(t^{-2}))^{-1}\mathrm{d}t$.
For any $\nu>0$ we have $\ell(t^{-2})>t^{2\nu}$, for $t>0$ small enough. Choosing $\nu$ such that $2(\rho+\nu)<1$ the considered integral converges.
\end{proof}

In the (critical) case $\rho=\frac{1}{2}$ we give an example of a recurrent random walk $(Z_{\tau (n)})$ with increments that have no finite first moment.
First we recall an important notion of $\alpha$-conjugate pairs from the theory of regular variation which we extract from \textsc{Doney} \cite{do82}.

For a given slowly varying function $\ell$ set $f(x)=x^\alpha \ell(x)$, with some $\alpha>0$. By \cite[Theorem 1.5.12]{bi87}, there is a regularly varying function $g$ of index $1/\alpha$ and such that $g(f(x))\sim x$ at infinity. Since $g$ varies regularly, $g(x)=x^{1/\alpha} \ell^{\ast}_\alpha (x)$, for some slowly varying $\ell^{\ast}_\alpha$. By definition, $\ell^{\ast}_\alpha$ satisfies 
\begin{align}\label{alpha_cojug}
 ( \ell(x) )^{1/\alpha } \ell^{\ast}_\alpha (x^\alpha \ell(x)) \to 1,\ \textrm{equivalently }\ (\ell^{\ast}_\alpha (x))^{\alpha}\ell (x^{1/\alpha}\ell^{\ast}_\alpha (x))\to 1,\ \ \text{as } x \rightarrow \infty.
\end{align}
The function $\ell^{\ast}_\alpha$ is called the $\alpha$-conjugate of the function $\ell$. 
The way to remember the meaning of $\ell^{\ast}_\alpha$ is that $y\sim x^\alpha \ell (x)$, when $x$ goes to infinity, if and only if $x\sim y^{1/\alpha}\ell^{\ast}_\alpha (y)$, as $y$ goes to infinity.
One easily checks that if 
\begin{align}
\lim_{x\rightarrow\infty}\frac{\ell(x)}{\ell(x^\alpha \ell(x))}= C(\alpha)>0 \quad &\text{then}\quad
\ell_\alpha^* (x)\sim \left(C(\alpha) \ell(x)\right)^{-\frac{1}{\alpha}}, \quad \text{as } x\rightarrow \infty. \label{ali: alpha conjugate }
\end{align}
This holds for many slowly varying functions, for example $\log x$, $\log \log x$ etc. 

\begin{example}
Let $(S_n)$ be an oscillating random walk such that its increment $Y$ belongs to the domain of attraction of the normal distribution. % $Y\in\mathrm{Dom}(2,0)$
It is known from \textsc{Feller} \cite[Ch. XVII, Sec. 5, Thm. 1a]{fe71II} that it holds if and only if the truncated variance of $Y$ is slowly varying, that is 
\begin{align}
\e \left( Y^2 \IndFun{(|Y|\leq y)} \right) \sim \frac{2}{\ell(y)},\quad\text{ as } y\rightarrow \infty, \label{ali: truncated variance}
\end{align}
for some slowly varying function $\ell $. We additionally assume that $\e(Y^2)=\infty$ and $\e(S_{\tau})<\infty$. Then the result by \textsc{Uchiyama} \cite[Thm. 1.2 and Rem. 2]{uc11} implies that
\begin{align*}
\p(\tau>n)&\sim \frac{1}{\sqrt{\pi} \e(S_\tau) n^{1/2}\ell^{\ast} (n)},\quad\text{as } n\rightarrow\infty,
\end{align*}
where $\ell^{\ast} = \ell^{\ast}_2$ is the $2$-conjugate of $\ell$ as defined in \eqref{alpha_cojug}. 
By \cite{vawa08} we obtain that
\begin{align*}
\p(\tau=n)&\sim \frac{1}{2\sqrt{\pi}\e(S_\tau) n^{3/2}\ell^{\ast} (n)},\quad\text{as } n\rightarrow\infty.
\end{align*}
Next, if we take a symmetric random walk $(Z_n)$ then 
Theorem \ref{THM_main} gives us
\begin{align*}
\p(Z_\tau =x)&\sim\frac{1}{\sqrt{2}\pi\e(S_\tau)  x^2\ell^*(x^2)},\quad\text{as } x\rightarrow\infty.
\end{align*}
As in the proof of Proposition \ref{pro: for which rho is Z_tau recurrent/tranisent}, 
\begin{align*}
H(x)=1-F(x)+F(-x)\sim \frac{1}{\sqrt{2}\pi \e(S_\tau) x \ell^*(x^2)},\quad\text{as } x\rightarrow\infty,
\end{align*}
and 
\begin{align*}
1-\phi(t)\sim\frac{\pi}{2} H(t^{-1})\sim \frac{t}{2^{3/2}\e(S_\tau)\ell^*(t^{-2})},\quad\text{as } t\rightarrow 0.
\end{align*}
Thus, to study recurrence of $(Z_{\tau (n)})$ we investigate convergence of the integral $\int_{0}^{\epsilon} \frac{1}{1-\phi(t)}\mathrm{d}t$ around zero. To simplify the calculations we restrict our attention to the specific choice of the slowly varying function in \eqref{ali: truncated variance} and for that reason we take
$\ell(x)=\log ^\eta x$, for $\eta\in \mathbb{R}$.
We immediately get by \eqref{ali: alpha conjugate } that 
$\ell^*(x)\sim 2^{\frac{\eta}{2}}\log^{-\frac{\eta}{2}}x$ at infinity.
Finally we are left with the integral $\int_0^\epsilon t^{-1} \log^{-\frac{\eta}{2} }(t^ {-2} )\mathrm{d}t$
which diverges for $\eta\leq2$ (and we get recurrence), whereas for $\eta>2$ it converges and implies transience.
Moreover, one easily verifies that $\e(|Z_\tau|)=\infty$ for $\eta\geq -2$ and $\e(|Z_\tau|)<\infty$ otherwise. Thus we have the following possibilities
\begin{itemize}
\item for $\eta<-2$ the random walk $(Z_{\tau (n)})$ is recurrent with finite first absolute moment,
\item for $-2\leq \eta\leq 2$ the random walk $(Z_{\tau (n)})$ is recurrent and $\e(|Z_\tau |)=\infty$,
\item for $2<\eta$ the random walk $(Z_{\tau (n)})$ is transient.
\end{itemize}
\end{example}

We end this section with a result concerning Theorem \ref{THM_main} when the increments of the random walk $(S_n)$ have finite second moment. 

\begin{proposition}\label{REmark}
Let $(S_n)$ be an oscillating random walk with the increment $Y$ having finite second moment and let $(Z_n)$ be a  centred and finite range random walk on $\mathbb{Z}$ independent of $(S_n)$. Then there is some $C>0$ such that 
\begin{align}\label{second_mom}
\lim_{y\rightarrow \infty} y^2\p \left(Z_{\tau}=y\right) =C,
\end{align}
and in this case $\e(|Z_{\tau}|)=\infty$. Equation \eqref{second_mom} holds also when $Y$ is symmetric and has a density.
\end{proposition}
\begin{proof}
The proof is similar to that of Theorem \ref{THM_main}, but in place of formula \eqref{ali: Vatutin Wachtel result} one uses the result by \textsc{\'{E}ppel'} \cite{ep79}, 
\begin{align*}
 \p(\tau=n)\sim cn^{-3/2}, \quad c>0, \ n\rightarrow \infty.
 \end{align*}
We prove that $\e(|Z_{\tau}|)=\infty$. We set $F(x)=\p(Z_{\tau}\leq x)$ and by \eqref{second_mom} we get that, for some $C_1>0$,
\begin{align}
1-F(n)\sim  \frac{C_1}{n},\quad  n\rightarrow \infty , \label{ali: behavior of 1-F(n)}
\end{align}
which means that $1-F(n)$ is regularly varying at infinity of index $-1$.
In view of symmetry this implies that $Z_{\tau}$ is in the domain of attraction of the Cauchy law. Therefore, if $\mathcal{C}$ is the distribution function of the Cauchy law, then there are sequences $b_n>0$ and $a_n>0$ such that $F^{\ast n}(b_nx+b_na_n)\to \mathcal{C}(x)$, for all $x$ as $n$ goes to infinity. We find the asymptotic behaviour of the normalizing sequence $(b_n)$.
It is known from \cite{gnko69} that $(b_n)$ satisfies $1-F(b_n)\sim \frac{C_2}{n}$ at infinity
and, by \eqref{ali: behavior of 1-F(n)}, we obtain that $b_n\sim C_3 n$ at infinity.
Finally, by \textsc{Tucker} \cite{tu75}, the integral $\int |x| \mathrm{d}F(x)$ is finite if and only if $\sum_{n\geq 1} n^{-2} b_n <\infty$ and the proof is finished.
\end{proof}

\section{Multidimensional Lindley process}\label{Sec_MLP}
We proceed to study recurrence of the LP in higher dimensions. We start by discussing the two-dimensional case 
for which we apply various probabilistic methods including arguments from renewal theory as well as Theorem \ref{THM_main} from Section \ref{Sec_Walks}.
In the last subsection we give a result about positive recurrence of the LP. To prove it we use a technique of local contractivity which has its roots in stochastic dynamical systems.

\subsection*{Two-dimensional LP}\hspace*{\fill} \\

Let $(W_n^i)$, $i=1,2$, be two Lindley processes as defined in \eqref{UniVarDef} with the underlying random walks $(S_n^i)$ with increments $Y^i$ which have distributions $\mu^i$ supported in $\mathbb{Z}$. We consider a process $(W_n^1, W_n^2)$ in the lattice quadrant $\mathbb{N}_0 \times \mathbb{N}_0$ and we assume that $\p \left (Y^1>0, Y^2>0 \right) >0$. Then the probability to reach $(0,0)$ from an arbitrary state after finitely many steps is positive. Thus, the origin and all the states that can be reached from it build a unique essential class. 
Without our assumption it may happen that some states will never be reached by the process even when $\mathrm{gcd} (\mathrm{supp}\, \mu^i)=1$, see the following example. We also emphasize that a precise description of essential classes in a general case is a very hard task.

\begin{example}\label{exa11}
Set $\mu=\frac{1}{4}\left(\delta_{ (-4,1)}+\delta_{(-3,2)}+\delta_{(1,-4)}+\delta_{ (2,-3)}\right)$. Then $\mathrm{gcd }( \mathrm{supp}\, \mu^i)=1$ and clearly the two coordinates of $(W_n^1,W_n^2)$ are transient and whence also the two dimensional process is transient. In this case every point in $\mathbb{N}_0\times \mathbb{N}_0$ will be visited at most one time a.s.
On the other hand,
setting $\mu=\frac{1}{4}\left(\delta_{ (-1,1)}+\delta_{(-1,2)}+\delta_{(1,-1)}+\delta_{ (2,-1)}\right) $ we also have $\mathrm{gcd }( \mathrm{supp}\, \mu^i)=1$ with positive recurrent coordinates and the two-dimensional LP $(W_n^1,W_n^2)$ will never reach $(0,0)$ in this case. We notice however that $(W_n^1,W_n^2)$ is positive recurrent in its essential class, cf. Theorem \ref{the: positive recurrence of multidimensional lp}.
\end{example}

We begin our discussion on recurrence with a very simple but fruitful lemma. 

\begin{lemma}\label{lem: W is recurrent if tilde W is recurrent}
Let $\bar{\tau}^1(n)$ be the $n$-th non-strict ascending ladder epoch of $(S_n^1)$. Assume that 
%$\p(Y^1>0, Y^2>0)>0$. If 
the first coordinate process $(W^1_n)$ and the projected process $(0,W^2_{\bar{\tau}^1(n)})$ are recurrent then the two-dimensional process $(W_n^1, W_n^2)$ is recurrent. 
If $W_n^1$ and $W^2_{\bar{\tau}^1(n)}$ are positive recurrent then $(W_n^1, W_n^2)$ is positive recurrent.
\end{lemma}

\begin{proof}
Let $T$ and $\widetilde{T}$ be the first return times to the point $(0,0)$ of $(W_n^1, W_n^2)$ and $(0,W^2_{\bar{\tau}^1(n)})$ respectively. By the assumption, $\widetilde{T}$ is a.s. finite. We claim that $T = \bar{\tau}^1(\widetilde{T})$. Indeed, we have
\begin{align*}
T&=\inf\{ n\geq 1: (W_n^1,W_n^2)=(0,0)\} 
= \inf\{ \bar{\tau}^1(n)\geq 1: (W_{\bar{\tau}^1(n)}^1,W_{\bar{\tau}^1(n)}^2)=(0,0)\}\\
&= \bar{\tau}^1\left( \inf\{ n\geq 1: (0,W^2_{\bar{\tau}^1(n)})=(0,0)\}\right)
= \bar{\tau}^1(\widetilde{T}),
\end{align*}
where we used the fact that $(W_n^1,W_n^2)$ attains the value $(0,0)$ only if $n\in \{\bar{\tau}^1(k): k\geq 0\}$. This in turn implies that $T$ is a.s. finite and we get the first part of the result.

For the positive recurrent case, we consider a random walk $\bar{\tau}^1(n)= \xi_1 +\ldots +\xi_{n}$ with independent increments $\xi_i= \bar{\tau}^1(i)-\bar{\tau}^1(i-1)$ which have the same law as $\bar{\tau}^1(1)$.
We build a new filtration $\{\mathcal{F}_n\}_{n\geq 1}$ given by
\begin{align*}
\mathcal{F}_n=\sigma \left( \bar{\tau}^1(1),\dotsc,\bar{\tau}^1(n), \left(Y^1_1,Y_1^2\right) ,\ldots, \left( Y^1_{\bar{\tau}^1(n)}, Y^2_{\bar{\tau}^1(n)}\right) \right)
\end{align*}
and notice that, since $\{\widetilde{T}\leq n\}\in \mathcal{F}_n$, $\widetilde{T}$ is a stopping time with respect to the filtration $\{\mathcal{F}_n\}_{n\geq 1}$. Moreover, the increments of the random walk $(\bar{\tau}^1(n))$ have the form
\begin{align*}
\xi_i= \inf_{k\geq 0} \left\{ Y^1_{\bar{\tau}^1(i-1)+1}+\ldots + Y^1_{\bar{\tau}^1(i-1)+k} >0 \right\}
\end{align*}
and therefore $\xi_n$ is independent of $\mathcal{F}_{n-1}$. This allows us to apply Wald's identity in the form $\e\bar{\tau}^1(\widetilde{T})=\e\bar{\tau}^1(1) \e \widetilde{T}<\infty$ which implies positive recurrence 
of $(W_n^1,W_n^2)$.
\end{proof}

In the next proposition we apply Lemma \ref{lem: W is recurrent if tilde W is recurrent} and combine it with an argument which comes from renewal theory.

\begin{proposition}\label{the: two-dimensional queuing process recurrent}
%Assume $\p(Y^1>0, Y^2>0)>0$ and 
Suppose that the random walks $(S_n^i)$, $i=1,2$, are independent and oscillating with increments $Y^i\in D(\alpha, \beta)$ satisfying $\rho_1+\rho_2>1$, where $\rho _i$ are defined in \eqref{ali: definition of rho}. 
Then the process $(W^1_{n}, W^2_{n})$ is null recurrent. 
%null because (\ora{because $W^1$ is only null recurrent, oscillating case)}
\end{proposition}

\begin{proof} 
Since $(S_n^1)$ is oscillating, $(W_n^1)$ is null recurrent.
Let $\tau^1(n)$ denote the $n$-th strict ladder epoch of $(S_n^1)$. We show that the Green function $G(0,0)$ of the process $(0,W^2_{\tau^1(n)})$ is infinite and whence it is a recurrent Markov chain. Evidently, this implies recurrence of the process $(0,W^2_{\bar{\tau}^1(n)})$ which in view of Lemma \ref{lem: W is recurrent if tilde W is recurrent} forces that $(W_n^1, W_n^2)$ is recurrent. Moreover, since $(W_n^1)$ is null recurrent, the two-dimensional process is also null recurrent as claimed.

An independence-based argument allows us to compute
\begin{align*}
G(0,0)&=\sum_{n=0}^{\infty} \p(W_{\tau^1(n)}^2=0)
=\sum_{n=0}^{\infty} \sum_{k=0}^{\infty} \p(W_k^2=0) \p(\tau^1(n)=k) \\
&= \sum_{k=0}^{\infty} \p(W_k^2=0)\sum_{n=0}^{\infty} \p(\tau^1(n)=k)
 = \sum_{k=0}^{\infty} \sum_{m=0}^k\p(\tau^2(m)=k)\sum_{n=0}^{k} \p(\tau^1(n)=k)\\
&= \sum_{k=0}^{\infty} u_k^1 u_k^2 ,\quad \mathrm{where }\ u_k^i = \sum_{n=0}^k \p(\tau^i(n)=k).
\end{align*}
The sequence $(u_k^1)$ is a renewal sequence, that is it satisfies the recursive equation
\begin{align*}
u^1_0=1,\quad  u^1_k&=\sum_{n=1}^k \p( \tau^1(1)=n) u^1_{k-n}.
\end{align*}
Since $\PP (S^1_1>0)>0$, we have $\mathrm{gcd\{k:\p(\tau^1(1)=k)>0 \}}=1$. Moreover, our assumption implies that \eqref{Tau_asympt} holds with some slowly varying function $\ell$. 
Therefore, applying the celebrated renewal theorem by \textsc{Garsia and Lamperti} \cite[Theorem 1.1]{gala62} we obtain that
\begin{align*}
\liminf_{k\rightarrow\infty}  \frac{u_k^1}{k^{\rho_1 -1}\ell (k)} =  \pi^{-1} \Gamma\left(\rho_1\right)\Gamma(1-\rho_1)\sin (\rho_1 \pi)  =C.
\end{align*}
Thus, for some $\epsilon>0$, $k_0>1$, and for all $k\geq k_0$, we have
\begin{align*}
u^1_k \geq (C-\epsilon) k^{\rho_1 -1} \ell (k)\geq (C-\epsilon )k^{\rho_1 -1 -\nu},\ \  \text{for any }\ \nu >0.
\end{align*}
Clearly, all the same holds for the sequence $(u^2_k)$ and whence for $C_1>0$ we have
$G(0,0)\geq  C_1 \sum_{k>k_0} k^{-2\nu + \rho_1+\rho_2 -2}$. Choosing $\nu$ such that $0<2\nu \leq \rho_1 +\rho_2 -1$ we conclude the claim.
\end{proof}

We also present a positive result in the case when $\rho_1=\rho_2=1/2$.
\begin{proposition} \label{pr: two-dimensional queuing process recurrent with rho=1/2 }
%Assume that $\p(Y^1>0, Y^2>0)>0$. 
If the random walks $(S_n^i)$, $i=1,2$, are independent, centered and with finite second moment then  $(W^1_{n}, W^2_{n})$ is null recurrent. 
%null (\ora{because each coordinate is only null recurrent, oscillating case})
\end{proposition}

\begin{proof}
The proof is similar as that of Proposition \ref{the: two-dimensional queuing process recurrent} but instead of asymptotics \eqref{Tau_asympt} we use
the result by \textsc{\'{E}ppel'} \cite[Theorem 1]{ep79}, that is 
\begin{align*}
\p(\tau^1(1)=n)&\sim K n^{-\frac{3}{2}},\quad \text{for }K>0,\ \text{as } n\rightarrow \infty . 
\end{align*}
This allows us to show that $G(0,0)$ is infinite and we again get the result.
\end{proof}

Our next result concerns recurrence of the two-dimensional process $(W_n,Z_n)$, where in the first coordinate $(W_n)$ is a LP with the underlying random walk $S_n=Y_1+\ldots +Y_n$ and the second coordinate $(Z_n)$ is a random walk on $\mathbb{Z}$ with increments $V_1,V_2,\ldots$. 

\begin{theorem}\label{the: recurrence of (W_n,Z_n)}
The two-dimensional process $(W_{n}, Z_{n})$ is recurrent in each of the following cases.\\
1. If $(W_{n})$ is positive recurrent and $(Z_n)$ is a centered random walk.\\
2. If $(S_n)$ is oscillating with the increment $Y\in D(\alpha,\beta)$ such that 
	$1/2< \rho <1$, and $(Z_n)$ is a symmetric finite range random walk independent of $(S_n)$.\\ 
3. If $(W_{n})$ is null recurrent with $\e(|Y|^2)<\infty$ and independent of 
	$(Z_n)$ which we assume to be a symmetric random walk of finite range.
\end{theorem}
\begin{proof}
Recall that the $k$-th return time of the first coordinate $(W_n)$ to $0$ is equal to the $k$-th non-strict ascending ladder epoch $\bar{\tau}(k)$ of the underlying random walk $(S_n)$. Thus, the return times to the origin of $(W_n,Z_n)$ are the same as for the induced random walk $(Z_{\bar{\tau}(n)})$. Moreover, the process $(Z_{\bar{\tau}(n)})$ is recurrent if the random walk $(Z_{\tau(n)})$ is recurrent, where $\tau (n)$ is the $n$-th strict ascending ladder epoch of $(S_n)$.

To prove the first assertion we notice that as $W_n$ is positive recurrent we known that $\e \tau <\infty$, where $\tau = \tau (1)$. We also notice that $\tau$ is a stopping time for the two-dimensional random walk $(S_n,Z_n)$ which implies that the event $\{\tau  \leq n\}$ is independent of $V_{n+1}$ and whence we are allowed to apply the Wald's identity in the form $\e Z_{\tau } =\e Z_1 \e \tau =0$.
Therefore $(Z_{\tau(n)})$ is recurrent. 

The second claim is a direct consequence of Proposition \ref{pro: for which rho is Z_tau recurrent/tranisent}.
In the last case we have $\e(Y)=0$ and, as follows by Proposition \ref{REmark}, the following asymptotic relation holds
\begin{align*}
y^2 \p(Z_{\tau }=y)&\rightarrow C>0, \quad \text{as } y\rightarrow \infty.
\end{align*}
Since $(Z_{\tau(n)})$ is a symmetric random walk, we conclude, for instance by \textsc{Spitzer} \cite[Sec. 8, E2]{sp76}, that it is recurrent.
\end{proof}

\subsection*{LP in higher dimensions}\hspace*{\fill} \\

We consider the multidimensional Lindley process $(W_n^1,\ldots, W_n^d)$ in $\mathbb{N}_0\times \dots \times \mathbb{N}_0$.
% starting from $(w_0^1,\ldots, w_0^d)$.
%, where $(W_n^i)$ are LP as defined at (\ref{UniVarDef}). 
The underlying random walks $(S_n^i)$ are governed by distributions $\mu^i$ which are supported by $\mathbb{Z}$ and such that $\p(Y^i>0)>0$.
% and $\mathrm{gcd }( \mathrm{supp}\, \mu^i)=1$.

% As in the two-dimensional case it may happen that some states will never be reached.

The following theorem treats positive recurrence of the multidimensional LP. Its proof uses elements from the theory of stochastic dynamical systems and thus we briefly present necessary definitions and facts, see \textsc{Peign\'{e} and Woess} \cite{pewo11I} and \cite{pewo11II} for more detailed description where in particular the substantial PhD work of \textsc{Benda} \cite{be98} is outlined.

Let $(X,d)$ be a proper metric space and denote by $\mathcal{C}=\mathcal{C}(X)$ the monoid of all continuous functions from $X$ to $X$ equipped with the topology of uniform convergence on compact sets. Fix a probability space $(\Omega ,\PP)$ and consider a sequence $(F_n)_{n\geq 1}$ of independent and identically distributed $\mathcal{C}$-valued random functions with a common distribution $\tilde{\mu}$.
The corresponding stochastic dynamical system $\omega \mapsto X_n^x(\omega)$ is given by
\begin{align*}
X_0^x = x,\quad X_n^x = F_n \circ \ldots \circ F_1(x), \quad x\in X,\ n\geq 1.
\end{align*}
For a LP on $\mathbb{N}_0\times \dots \times \mathbb{N}_0$ we have $F_n(x) = \max\{x-Y_n,0\}$ and these mappings are contractions so that we can restrict our attention to the set $\mathcal{L}\subset \mathcal{C}$ of all Lipschitz mappings with Lipschitz constants $\leq 1$. Notice that if $\mu$ is the distribution of $Y_n$ then $\tilde{\mu}$ is the image of $\mu$ under $y\mapsto f_y,\, f_y(x) = \max \{x-y,0\}$.
A stochastic dynamical system is called 
conservative if 
\begin{align*}
\PP \left( \liminf_{n\to \infty} d(X_n^x, x)<\infty \right) =1,\quad \textrm{for\ every}\ x\in X,
\end{align*}
and it is locally contractive if for every $x\in X$ and every compact set $K\subset X$,
\begin{align*}
\PP \left( d(X_n^x, X_n^y) \IndFun{K}(X_n^x)\rightarrow 0,\text{ for\ all } y\in X \right) =1.
\end{align*} 
For $\omega \in\Omega$ we consider the set $L^x(\omega)$ of all accumulation points of the sequence $(X_n^x(\omega))_{n\geq 0}$ in $X$. The following lemma allows us to show that there is only one essential class for the multidimensional LP, see \cite[Lemma (2.5)]{pewo11I}.

\begin{lemma}\label{Lema_2}
For a conservative and locally contractive stochastic dynamical system, there exists a set $L\subset X$ such that 
\begin{align*}
\PP \left( L^x(\cdot) = L,\, \mathrm{for\ all}\ x\in X \right) =1.
\end{align*}
\end{lemma}

We now prove the main theorem of this section.

\begin{theorem}\label{the: positive recurrence of multidimensional lp}
Suppose that each of the Lindley processes $(W_n^i)$ is positive recurrent. Then there exists a unique invariant probability measure of the process $(W_n^1, \dotsc, W_n^d)$.
The stationary measure is the distribution of a random variable $U$ which is the limit of the backward process 
\begin{align}\label{Backward}
F_1\circ \dots \circ F_n (x)\xrightarrow{a.s} U,\quad \textrm{for\ every} \ x\in \mathbb{R}^d,
\end{align}
where $F_n(x) = \max\{x-Y_n,0\}$. In particular, 
the multidimensional process $(W_n^1, \dotsc, W_n^d)$ is positive recurrent in its unique essential class.
\end{theorem}

\begin{proof}
We start by showing that the LP $(W_n^i)$ is locally contractive  in each coordinate. 
We set $f_y(x)=\max\{x-y, 0\}$ and consider random contractions $F_n^i=f_{Y^i_n}$ with law $\tilde{\mu }^i$ which is the image of $\mu^i$ under the mapping $f_y$. Let $\mathfrak{S}^i$ be the closed sub-semigroup of $\mathcal{L}$ generated by $\mathrm{supp}(\tilde{\mu}^i)$.
Our aim is to show that there is a constant function in $\mathfrak{S}^i$ and this, in view of \cite[Corollary 4.4]{pewo11I}, will force local contractivity. The claim follows by the assumption $\PP (Y^i>0)>0$. Indeed, there is $y>0$ such that for any $x\in \mathbb{R}$ there is $N_x>1$ such that for all $n\geq N_x$ we obtain that the $n$-fold composition $f_y\circ \dots \circ f_y (x)\equiv 0$,
and thus the null-function lies in $\mathfrak{S}^i$ as desired.

Next, by positive recurrence, each $(W_n^i)$ has a unique invariant probability measure, say $\nu^i$. This together with local contractivity imply that for any starting point $x^i$ we have the a.s.-convergence of the backward process $F^i_1\circ\dots \circ F^i_n(x^i)\rightarrow U^i$, where $U^i$ is a random variable with distribution $\nu^i$. This goes back to \textsc{Leguesdron} \cite{le89}, compare with \cite[Prop. 2.6]{klwo16}. Since we have convergence of all coordinates, we get \eqref{Backward}.
Applying Furstenberg's contraction principle, see \cite[Prop. 1.3]{pewo11I}, we conclude that the distribution $\nu$ of the random vector $(U_1,U_2, \dotsc ,U_d)$ is the unique invariant probability measure for $(W_n^1,\dotsc,W_n^d)$ which is equivalent to positive recurrence.

It is left to show that there is only one essential class. Indeed, by the very definition, local contractivity of the coordinates implies that the multidimensional process $(W_n^1,\ldots ,W_n^d)$ is locally contractive as well. Since we have proved it is recurrent, it must be conservative. In our case, the Lindley process lives on the grid and thus the deterministic set $L\subset \mathbb{N}_0\times \dots \times \mathbb{N}_0$ from Lemma \ref{Lema_2} is such that, independently of the starting point,
\begin{align*}
\PP \left( (W_n^1,\ldots ,W_n^d) = l\ \mathrm{for\ infinitely\ many}\ n \right) =1,\quad \mathrm{for\ every}\ l\in L.
\end{align*}
We clearly conclude that $L$ is the unique essential class of $(W_n^1,\ldots ,W_n^d)$.
\end{proof}

\begin{remark}
To prove the existence of positive recurrent states of $(W_n^1,\ldots ,W_n^d)$ there is a simple argument which was presented to us by Nina Gantert and mentioned already in the context of RRW in \cite[Remark (4.10)]{klwo16}. However, this argument yields no understanding of the number of essential classes and their absorption properties.
The use of local contractivity leads to an answer and additional insight, namely a.s. convergence of the backward process.
\end{remark}

\subsection*{Acknowledgement} We are grateful to Prof.~Marc Peign\'{e} (Univ. Tours) and Prof.~Wolfgang Woess (TU Graz) for having brought our attention to this topic and for many stimulating discussions. We thank Dr.~Sebastian M\"{u}ller (Univ. Aix-Marseille) for fruitful communication.
We also thank the referee for helpful suggestions that improved readability of the paper.

\bibliographystyle{abbrv}
\bibliography{mybib_diss}   

\begin{thebibliography}{10}

\bibitem{al02}
G.~K. Alexopoulos.
\newblock Random walks on discrete groups of polynomial volume growth.
\newblock {\em Ann. Probab.}, 30(2):723--801, 2002.

\bibitem{as00}
S.~Asmussen.
\newblock {\em Applied {P}robability and {Q}ueues}.
\newblock Springer, 2000.

\bibitem{be98}
M.~Benda.
\newblock {\em Schwach kontraktive dynamische Systeme}.
\newblock PhD thesis, Ludwig-maximilians-Universit\"{a}t M\"{u}nchen, 1998.

\bibitem{cygan1}
A.~Bendikov and W.~Cygan.
\newblock Alpha-stable random walk has massive thorns.
\newblock {\em Colloq. Math.}, 138:105--129, 2015.

\bibitem{BCT}
A.~Bendikov, W.~Cygan, and B.~Trojan.
\newblock Limit theorems for random walks.
\newblock {\em Stochastic Process. Appl.}, 127:3268--3290, 2017.

\bibitem{bi87}
N.~Bingham, C.~Goldie, and J.~Teugels.
\newblock {\em Regular {V}ariation}.
\newblock Cambridge {U}niversity {P}ress, 1987.

\bibitem{Borovkov1976}
A.~A. Borovkov.
\newblock {\em Stochastic {P}rocesses in {Q}ueueing {T}heory}.
\newblock Springer-Verlag, New York-Berlin, 1976.

\bibitem{chfu51}
K.~L. Chung and W.~H.~J. Fuchs.
\newblock On the distribution of values of sums of random variables.
\newblock {\em Mem. Amer. Math. Soc.}, No. 6:12, 1951.

\bibitem{Diaconis}
P.~Diaconis and D.~Freedman.
\newblock Iterated random functions.
\newblock {\em SIAM Rev.}, 41(1):45--76, 1999.

\bibitem{do82}
R.~A. Doney.
\newblock On the exact asymptotic behaviour of the distribution of ladder
  epochs.
\newblock {\em Stochastic Process. Appl.}, 12(2):203--214, 1982.

\bibitem{do95}
R.~A. Doney.
\newblock Spitzer's condition and ladder variables in random walks.
\newblock {\em Probab. Theory Related Fields}, 101(4):577--580, 1995.

\bibitem{ep79}
M.~S. \`Eppel'.
\newblock A local limit theorem for the first passage time.
\newblock {\em Sib. Math. J.}, 20(1):181--191, 1979.

\bibitem{espera13}
R.~Essifi, M.~Peign\'e, and K.~Raschel.
\newblock Some aspects of fluctuations of random walks on {$\mathbb{R}$} and
  applications to random walks on {$\mathbb{R}^+$} with non-elastic reflection
  at $0$.
\newblock {\em ALEA Lat. Am. J. Probab. Math. Stat.}, 10(2):591--607, 2013.

\bibitem{fe71II}
W.~Feller.
\newblock {\em An {I}ntroduction to {P}robability {T}heory and its
  {A}pplications}, volume~II.
\newblock Wiley, New York, 1971.

\bibitem{gala62}
A.~Garsia and J.~Lamperti.
\newblock A discrete renewal theorem with infinite mean.
\newblock {\em Comment. Math. Helv.}, 37:221--234, 1962/1963.

\bibitem{gnko69}
B.~V. Gnedenko and A.~N. Kolmogorov.
\newblock {\em Limit {D}istributions for {S}ums of {R}andom {V}ariables}.
\newblock Addison-Wesley publishing company, 1968.

\bibitem{gu88}
A.~Gut.
\newblock {\em Stopped Random Walks}.
\newblock Springer, 1988.

\bibitem{Kendall}
D.~G. Kendall.
\newblock Some problems in the theory of queues.
\newblock {\em J. Roy. Statist. Soc. Ser. B.}, 13:151--173; discussion:
  173--185, 1951.

\bibitem{Kiefer_Wolfowitz}
J.~Kiefer and J.~Wolfowitz.
\newblock On the theory of queues with many servers.
\newblock {\em Trans. Amer. Math. Soc.}, 78:1--18, 1955.

\bibitem{klwo16}
J.~Kloas and W.~Woess.
\newblock Multidimensional random walk with reflections.
\newblock 2016.
\newblock preprint, arXiv:1704.06055v1.

\bibitem{lali10}
G.~F. Lawler and V.~Limic.
\newblock {\em Random {W}alk: {A} {M}odern {I}ntroduction}.
\newblock Cambridge Studies in Advanced Mathematics, 2010.

\bibitem{le89}
J.~P. Leguesdron.
\newblock Marche al\'eatoire sur le semi-groupe des contractions de
  $\mathbb{R}^d$. {C}as de la marche al\'eatoire sur {$\mathbb{R}_+$} avec choc
  \'elastique en z\'ero.
\newblock {\em Ann. Inst. H. Poincar\'e Probab. Statist.}, 25(4):483--502,
  1989.

\bibitem{Lindley}
D.~V. Lindley.
\newblock The theory of queues with a single server.
\newblock {\em Proc. Cambridge Philos Soc.}, 48:277--289, 1952.

\bibitem{pewo08}
M.~Peign\'e and W.~Woess.
\newblock On {R}ecurrence of {R}eflected {R}andom {W}alk on the {H}alf-line.
\newblock 2008.
\newblock unpublished manuscript, arXiv:0612306v1.

\bibitem{pewo11I}
M.~Peign\'e and W.~Woess.
\newblock Stochastic dynamical systems with weak contractivity properties {I}.
  {S}trong and local contractivity.
\newblock {\em Colloq. Math.}, 125(1):31--54, 2011.

\bibitem{pewo11II}
M.~Peign\'e and W.~Woess.
\newblock Stochastic dynamical systems with weak contractivity properties {II}.
  {I}teration of {L}ipschitz mappings.
\newblock {\em Colloq. Math.}, 125(1):55--81, 2011.

\bibitem{pi68}
E.~J.~G. Pitman.
\newblock On the behavior of the characteristic function of a probability
  distribution in the neighborhood of the origin.
\newblock {\em J. Austral. Math. Soc.}, 8:423--443, 1968.

\bibitem{sp76}
F.~Spitzer.
\newblock {\em Principles of {R}andom {W}alk}.
\newblock Springer, 1976.

\bibitem{tu75}
H.~G. Tucker.
\newblock On moments of distribution functions attracted to stable laws.
\newblock {\em Houston J. Math.}, 1(1):149--152, 1975.

\bibitem{uc11}
K.~Uchiyama.
\newblock A note on summability of ladder heights and the distributions of
  ladder epochs for random walks.
\newblock {\em Stochastic Process. Appl.}, 121(9):1938--1961, 2011.

\bibitem{vawa08}
V.~A. Vatutin and V.~Wachtel.
\newblock Local limit theorem for ladder epochs.
\newblock 2007.
\newblock unpublished manuscript, arXiv:0701914.

\bibitem{vawa2}
V.~A. Vatutin and V.~Wachtel.
\newblock Local probabilities for random walks conditioned to stay positive.
\newblock {\em Probab. Theory Relat. Fields}, (143):177--217, 2009.

\end{thebibliography}

\end{document}